\documentclass{amsart}

\usepackage[english]{babel}
\usepackage{csquotes}
\usepackage[all]{xy}
\usepackage{enumitem}
\usepackage{amssymb}
\usepackage{mathrsfs}
\usepackage[backend=biber,style=alphabetic,
  doi=false,
  url=false]{biblatex}
\usepackage{bm}
\usepackage{algorithm, algorithmicx, algpseudocode}

\usepackage{xcolor}
\usepackage[
  colorlinks=true,
  linkcolor=gray,
  citecolor=gray,
  urlcolor=blue
]{hyperref}

\usepackage[a4paper,
  left=30mm,
  right=30mm,
  top=28mm,
  bottom=32mm
]{geometry}

\newtheorem{theorem}{Theorem}[section]
\newtheorem{lemma}[theorem]{Lemma}
\newtheorem{proposition}[theorem]{Proposition}
\newtheorem{corollary}[theorem]{Corollary}

\theoremstyle{definition}
\newtheorem{definition}[theorem]{Definition}
\newtheorem{example}[theorem]{Example}

\theoremstyle{remark}

\numberwithin{equation}{section}

\newcommand{\RR}{\mathbb{R}}
\newcommand{\CC}{\mathbb{C}}
\newcommand{\KK}{\mathbb{K}}
\newcommand{\NN}{\mathbb{N}}

\newcommand{\bA}{\mathbb{A}}

\newcommand{\HH}{\mathcal{H}}
\newcommand{\cI}{\mathcal{I}}
\newcommand{\cJ}{\mathcal{J}}

\newcommand{\rar}{\rightarrow}

\newcommand\rest[1]{\raisebox{-.5ex}{$|$}_{#1}}

\newcommand{\ul}[1]{\underline{#1}}

\newcommand{\Sym}{\textnormal{Sym}}
\newcommand{\inv}{\textnormal{inv}}
\newcommand{\Frac}{\textnormal{Frac}}
\newcommand{\id}{\textnormal{id}}
\newcommand{\GL}{\textnormal{GL}}
\newcommand{\OO}{\textnormal{O}}
\newcommand{\SO}{\textnormal{SO}}
\newcommand{\BB}{\textnormal{B}}
\newcommand{\DD}{\textnormal{D}}
\newcommand{\sB}{\mathscr{B}}
\newcommand{\sgn}{\textnormal{sgn}}

\newcommand\com[1]{}

\title[Rational Invariants of even degree polynomials]{Rational Invariants of even degree polynomials under the orthogonal group}
\author{Henri Breloer}
\address{Department of Mathematics and Statistics, UiT - the Arctic University of Norway, 9037 Tromsø, Norway}
\email{henri.l.breloer@uit.no}
\date{\today}

\begin{document}
\begin{abstract}
     We construct a generating set of rational invariants for the action of the orthogonal group $\OO_n$ on the space $\RR[x_1,\dots,x_n]_{2d}$ of real homogeneous polynomials of even degree $2d$. This builds on a paper that addressed the case $n=3$. Such a generating set has applications in separating general $\OO_n$ orbits. In the case of higher $n$, we show that the main obstruction in finding a set of rational invariants that separate all orbits lies in a surprising connection to the graph isomorphism problem, a classical problem of computer science. We also obtain a partial result for when it is possible to obtain an algebraically independent generating set.
\end{abstract}
\maketitle

\section{Introduction}

The space of homogeneous polynomials of degree $2d$ in $n$ variables $\RR[x_1,\dots,x_n]_{2d}$ is naturally isomorphic to the space of real symmetric tensors of order $2d$ over the vector space $\RR^n$. There are many examples in physics where this kind of tensors arise and a natural action of the orthogonal group $\OO_n$ on the underlying vector space $\RR^n$ is given. One example where the underlying vector space can have dimension $n>3$ are the renormalization group fixed points in quantum field theory of $n$-scalar fields (\cite{Mic84,RS19}).

Given such a linear group action on a vector space, a fundamental question of invariant theory is to describe its invariant algebra by giving a generating set. This is a well developed field with many tools, both theoretical and computational, see e.g. \cite{DK15}. Nevertheless, this problem is often intractable.

An alternative is given by understanding the field of rational invariants through a concrete generating set. This question is interesting in its own right, but it also allows for practical applications. It is a well known fact, that a generating set for the rational invariants separates general orbits \cite{PV94, Ros56}. This allows one to separate orbits arising from real world data, where the restriction to general polynomials rarely matters.

In this article, we construct a set of $\binom{n+2d-1}{2d}-\binom{n}{2}+1$ rational invariants which generate the field of rational invariants for the action of $\OO_n$ on $\RR[x_1,\dots,x_n]_{2d}$. We mostly work on a so called slice and give an isomorphism between the rational invariants. This leads to algorithm \ref{eval_alg}, which evaluates the rational invariants at a point.
Our generating set contains one element more than the transcendence degree of the field of rational invariants over $\RR$.

The question of rational $\OO_3$ invariants on the space of polynomials in three variables was addressed in \cite{GHP19}. We aim to push the methods used in this article to the greatest generality, finding a generating set for general $n$. In particular, working with a general number of variables forces us to work with a different equivariant basis and reveals an interesting connection to the graph isomorphism problem from computer science.

The slice we use is fundamentally one of the space of quadratic forms, therefore our method only works for a space of polynomials of even degree. In a recent article, \cite{HJ24} worked with a different slice for the action of the orthogonal group. As shown in the reference, the different slice makes it possible to work with odd degree ternary forms. This possibly gives an avenue to finding rational invariants for the space of polynomials of a given odd degree in general.

\subsection*{Structure of the paper}
We introduce the necessary preliminaries, in particular the slice method, in section \ref{sec2}.
In section \ref{sec3}, we construct such a slice $\Lambda_{2d}^n$ for the action of $\OO_n$ on the space of degree $2d$ polynomials and give a basis which is well-behaved under the stabilizer subgroup for the slice.
This allows us to manually construct a generating set of invariants for a subspace of our slice in section \ref{sec4} which behaves similarly to the space of weighted graphs on $n$ vertices. In section \ref{sec5}, we systematically construct a complete set of rational invariants out of the previous invariants. 
This implicitly gives invariants for the whole space of even degree polynomials through the slice lemma. We make this explicit by an evaluation algorithm in section \ref{sec6}.
To conclude the paper, we briefly discuss the possibility of an algebraically independent generating set and the action of $\SO_n$ in  section \ref{sec7}.
 
\textbf{Acknowledgments.\qquad}
This work has been supported by European Union’s HORIZON-MSCA-2022-DN-JD programme under the Horizon Europe (HORIZON) Marie Skłodowska-Curie Actions, grant agreement 101120296 (TENORS).
The author would also like to thank Ben Blum-Smith for spotting an error in an earlier version of this article.

\section{Preliminaries}\label{sec2}
Let $G\subset \GL_n(\RR)$ be a linear algebraic group. We fix a linear action of $G$ on a real vector space $V$ denoted by
\[
G\times V \longrightarrow V,\quad (g,v)\mapsto gv.
\]
Many techniques of algebraic geometry only work for an algebraically closed base field. This makes it necessary to work with the complexification of the group and vector space.
The complexification of a vector space is given by the tensor product $V(\CC) = V\otimes_\RR\CC$. For a linear algebraic group $G\subset\GL_n(\RR)$ given as the vanishing of some polynomial functions on $\GL_n(\RR)$, the corresponding complexification is given by the zero-set of the same polynomial functions on $\GL_n(\CC)$.

\begin{example}
    The orthogonal group $\OO_n \subset \GL_n(\RR)$ is a linear algebraic group. It is defined as the real zero set of the polynomial equations induced by $g^T g = \id$ where $g^T$ is the transpose of the matrix $g$. Its complexification $\OO_n(\CC) \subset \GL_n(\CC)$ is the set of complex invertible matrices with $g^Tg = \id$.

    The special orthogonal group $\SO_n$ consists of all orthogonal matrices $g\in \OO_n$ with $\det(g) = 1$. It is the connected component of $\OO_n$ containing the identity element.
\end{example}

\subsection{Invariants and rational invariants}
Let $\KK$ be $\RR$ or $\CC$ and $V$ be a finite dimensional $\KK$ vector space.
We write $\KK[V]$ for the ring of polynomial functions on $V$. After picking a basis $v_1,\dots,v_N\in V$ and taking $y_1,\dots,y_N$ for the dual basis, we can explicitly write $\KK[V]$ as the polynomial ring $\KK[y_1,\dots,y_N]$.
We write $\KK(V)$ for the field of rational functions on $V$. It is simply the fraction field of the ring of polynomial functions on $V$, $\KK(V) := \Frac(\KK[V])$.
An element $f\in\KK(V)$ is a rational function $f = \frac{p_1}{p_0}: V\dashrightarrow \KK$ with $p_1,p_0\in \KK[V]$, $p_0 \neq 0$. It is only defined on a general point of $V$ which is indicated by the dashed arrow.

\begin{definition}
    A statement about a point in the $\KK$ vector space $V$ is said to hold for a general point, if it holds for all points $v\in U \subset V$ of a non-empty Zariski open $U\subset V$. Alternatively, there exists a non-zero polynomial function $p_0: V \rar \KK$ such that the property holds for $v\in V$ whenever $p_0(v) \neq 0$.
\end{definition}

\begin{definition}
    The $G$-action on the vector space $V$ induces an action on the ring of polynomial functions on $V$, denoted $\KK[V]$. For $g\in G$, $p\in \KK[V]$ and $v\in V$ the action is defined by $(g\cdot p)(v) = p(g^{-1}v)$. The ring of invariants for the group action of $G$ on $V$ is given by all polynomial functions invariant under the group action. We write
    \[
    \KK[V]^{G} := \{ p\in\KK[V] \;|\; g\cdot p = p \;\; \forall g\in G \}.
    \]
    Similarly, we have an induced $G$-action on the field of rational functions on $V$. For $g\in G$, $f\in \KK(V)$ and a general point $v\in V$ the action is defined by $(g\cdot f)(v) = f(g^{-1}v)$. We define the field of rational invariants correspondingly as
    \[
    \KK(V)^G := \{ f\in\KK(V) \;|\; g\cdot f = f \;\;\forall g\in G\}.
    \]
\end{definition}

\begin{lemma}\label{fininte-group-frac-field}
    Let $V$ and $G$ be as above. If $G$ is a finite group, then $\KK(V)^{G} = \Frac(\KK[V]^{G})$.
\end{lemma}
This is a special case of \cite[Theorem 3.3]{PV94}.

\begin{definition}
    Let $V$ be a finite dimensional $\KK$ vector space and suppose we have a linear group action of a linear algebraic group $G$ on $V$. We say that a finite subset $\mathcal{I}\subset\mathbb{K}[V]^G$ separates orbits, if for all $v,v'\in V$: $p(v) = p(v')$ for all $p\in \cI$ only if there exists a $g\in G$ such that $gv = v'$.
    
    We say that a finite subset $\mathcal{J}\subset \mathbb{K}(V)^{G}$ separates general orbits if there exists a nonempty Zariski-open $U\subset V$ on which all $f\in \cJ$ are defined, such that for all $v,v' \in U$: $f(v) = f(v')$ for all $f\in \mathcal{J}$ only if there exists a $g\in G$ with $gv = v'$. 
\end{definition}

\begin{lemma}\label{general-orbit-sep}
    Let $V$ and $G$ be as above. A finite subset $\mathcal{J}\subset \CC(V)^G$ separates general orbits if and only if it already generates the field of rational invariants, $\CC(V)^G = \CC(f\mid f\in\mathcal{J})$.
\end{lemma}
This is a special case of \cite[Lemma 2.1]{PV94}.

\begin{lemma}\label{orbit-sep}
    Let $V$ be as above and let $G$ be a finite group with linear action on $V$. If a finite subset $\mathcal{I}\subset\mathbb{K}[V]^G$ generates the ring of invariants as a $\mathbb{K}$-algebra then it separates the $G$-orbits.
\end{lemma}

\begin{proof}
    Take any $v,w\in V$ not lying on the same orbit. Denote by $Gv$ the orbit of $v$ in $V$. It is finite and does not contain $w$ thus we find a polynomial $\hat{p}\in \mathbb{K}[V]$ such that $\hat{p}(w) = 0$ and $\hat{p}(v') \neq 0$ for all $v'\in Gv$. Let $\hat{p}^G := \prod_{g\in G} g\cdot \hat{p}$. Through this construction $\hat{p}^G\in \mathbb{K}[V]^G$ and $\hat{p}^G(w) = 0$ while $\hat{p}^G(v) \neq 0$. The invariant $\hat{p}^G$ separates $v$ and $w$.

    Now assume there exists a finite generating set $\mathcal{I}\subset \KK[V]^G$ which does not separate orbits. It follows that there exist $v,w\in V$ which do not lie on the same orbit but still $p(v) = p(w)$ for all $p\in \mathcal{I}$. But now the separating invariant $\hat{p}^G$ constructed above is a polynomial in the $p\in \mathcal{I}$ thus $\hat{p}^G(v) = \hat{p}^G(w)$ which is a contradiction. $\mathcal{I}$ must already separate orbits.
\end{proof}

We can now describe the problem this article will deal with. We denote the vector space of degree $2d$ homogeneous polynomials in $n$ variables over $\RR$ by
\[
V_{2d}^n = \RR[x_1,\dots,x_n]_{2d}.
\]
We can also view it as a symmetric power of a real $n$-dimensional vector space, $V_{2d}^n = \Sym^{2d}(\RR^n)$. In this way the action of the orthogonal group $\OO_n \subset \GL(\RR^n)$ extends naturally to $V_{2d}^n$.

Our goal in this article is to find a generating set for the field of rational invariants $\RR(V_{2d}^n)^{\OO_n}$. We describe these in Corollary \ref{main_corollary}.

\subsection{The slice method}
The main method to find these invariants is to reduce to the case of a finite group. This is done using the Slice Lemma.
A corollary of the no-name lemma shows that $\RR(V_{2d}^n)^{\OO_n}/\RR$ is a purely transcendental field extension for odd $n$. We show this in Theorem \ref{odd_purety}. This is non-constructive, however, and does not help to find a minimal generating set for this case. 

\begin{definition}\label{slice_def}
    Let $V$ be a finite dimensional real vector space and let $G$ be a real algebraic group with linear action on $V$. Consider a linear subspace $\Lambda\subset V$.
    We define the \emph{stabilizer subgroup}
    \[
    \BB := \{ g\in G\;|\; gs\in\Lambda\;\forall s \in \Lambda \} \subset G.
    \]
    If the following conditions hold
    \begin{enumerate}[label=(\roman*)]
        \item For a general point $v\in V$ there exists a $g\in G$ such that $gv\in\Lambda$.
        \item For a general point $s\in \Lambda(\CC)$ it holds that if $gs \in \Lambda(\CC)$ for $g\in G(\CC)$, then $g \in \BB(\CC)$ already,
    \end{enumerate}
    then we call $\Lambda\subset V$ a slice for the group action of $G$ on $V$ with stabilizer $\BB$.
\end{definition}

\begin{theorem}[Slice Lemma]\label{slice_thm}
    Let $G$ be a real algebraic group with a linear action on the finite dimensional $\RR$-vector space $V$. Let $\Lambda \subset V$ be a slice for this group action with stabilizer $\BB$. Then there exists an $\RR$-algebra isomorphism
    \[
    \rho: \RR(V)^G \xrightarrow{\cong} \RR(\Lambda)^\BB,\quad f \mapsto f\rest{\Lambda}
    \]
     between the rational invariants of the action of $G$ on $V$ and the action of $\BB$ on $\Lambda$.
\end{theorem}

This is a special case of \cite[Theorem 3.1]{CTS07} where a proof in a more general algebro-geometric setup can be found. The inverse of the isomorphism is given by mapping an $f\in \RR(\Lambda)^\BB$ to the rational map $(v \mapsto f(gv))$ where $g\in G$ is such that $gv\in \Lambda$. Note that this map is only defined on a general point of $V$. The slice method goes back to Seshadri \cite{Ses62}.

\subsection{The harmonic decomposition}
From here on out, we assume that $n \geq 3$ to avoid too many case distinctions. We denote by
\[
\Delta := \sum_{i=1}^n \frac{\partial^2}{\partial x_i^2}
\]
the \emph{Laplacian operator} and by $|\bm{x}| := \sqrt{x_1^2+\dots+x_n^2}$ the Euclidean norm of the vector of variables. Note that $|\bm{x}|^2 \in V_2^n$ is $\OO_n$ invariant.

\begin{definition}
    A polynomial whose Laplacian is zero is called a \emph{harmonic polynomial}. We write
    \[
    \mathcal{H}_{2d}^n := \big\{v\in V_{2d}^n \mid \Delta(v) = 0 \big\}
    \]
    for the space of degree $2d$ homogeneous harmonic polynomials.
\end{definition}

An orthogonal transformation commutes with the Laplacian, thus $\mathcal{H}_{2d}^n$ is an $\OO_n$-invariant subspace of $V_{2d}^n$. It is a well known fact that the space of homogeneous harmonic polynomials $\HH_k^n$ is even an irreducible $\OO_n$ representation, see e.g. \cite[Theorem 2.12]{CW68}. Given an $\OO_n$ invariant subspace $W\subset V_{2d}^n$, we can obtain an invariant subspace of $V_{2d+2k}^n$ by multiplication with $|\bm{x}|^{2k}$. It is given by
\[
    |\bm{x}|^{2k}W := \{ |\bm{x}|^{2k}v\mid v \in W \} \subset V_{2d+2k}^n.
\]
This lets us describe the \emph{harmonic decomposition} of $V_{2d}^n$ into the direct sum of $\OO_n$ invariant subspaces
\begin{equation}\label{harmonic_decomp}
    V_{2d}^n = |\bm{x}|^{2d}\HH_{0}^n \oplus |\bm{x}|^{2d-2}\HH_{2}^n \oplus \dots \oplus \HH_{2d}^n.
\end{equation}
In particular, we have the recursive description $V_{2d}^n = |\bm{x}|^2 V_{2d-2}^n \oplus \HH_{2d}^n$ and Schur's lemma gives the following result for the image of the Laplacian on subspaces of $V_{2d+2k}^n$.

\begin{lemma}\label{harmonic_laplacian}
For any $d\geq 0$, $k\geq 1$, we have
    \[
    \Delta(|\bm{x}|^{2k}\HH_{2d}^n) = |\bm{x}|^{2k-2}\HH_{2d}^n.
    \]
\end{lemma}

\section{Equivariant basis of a slice}\label{sec3}
In this section, we will construct a slice $\Lambda_{2d}^n$ for $V_{2d}^n$ under $\OO_n$-action for $n\geq 3$. This part is a direct generalization of the construction in \cite{GHP19}. The stabilizer of this slice turns out to be $\BB_n$, the group of \emph{signed permutation matrices}. A signed permutation matrix is a matrix where each row and column contains exactly one non-zero element which is either $1$ or $-1$. To find useful invariants for the slice, we then continue to construct a symmetry adapted basis and an orbit-type decomposition for the slice.

\subsection{A slice for the orthogonal group} 
\begin{definition}
    Let $\Lambda_2^n \subset V_2^n$ be the vector space of degree 2 homogeneous polynomials in $n$ variables without mixed terms,
    \[
    \Lambda_2^n := \Big\{ \sum_{i=1}^n\lambda_i x_i^2 \;|\; \lambda_i\in \RR \Big\} \subset V_{2}^n.
    \]
    We define $\Lambda_{2d}^n$ recursively by $\Lambda_{2d}^n := |\bm{x}|^2\Lambda_{2d-2}^n \oplus \HH_{2d}^n$.
\end{definition}

\begin{proposition}
    The subspace $\Lambda_{2d}^n \subset V_{2d}^n$ is a slice for the group action of $\OO_n$ with stabilizer $\BB_n \subset \OO_n$ of signed permutation matrices. In particular, there is a one-to-one correspondence of rational invariants
    \[
    \rho: \RR(V_{2d}^n)^{\OO_n} \xrightarrow{\quad\cong\quad} \RR(\Lambda_{2d}^n)^{\BB_n}
    \]
    given by restriction.
\end{proposition}

\begin{proof}
    We begin with the case $d = 1$.
    It is a well known fact, that the space of quadratic forms is isomorphic to the space of symmetric matrices. Concretely, we have an isomorphism
    \[
    \Phi:  \text{Sym}^2(\RR^n) \xrightarrow{\ \ \cong\ \ } V_2^n \qquad M \mapsto \bm{x}^T M \bm{x}
    \]
    mapping a real symmetric $n\times n$ matrix to its corresponding quadratic form in $V_2^n$.  
    Let $v\in V_{2}^n$ be any degree $2$ polynomial with corresponding matrix $M$. The Spectral Theorem for symmetric matrices
    shows that there exists an element $g\in\OO_n$ such that $g^TMg$ is diagonal. We have
    \[
    \Phi^{-1}(gv) =\Phi^{-1}((g\bm{x})^TM(g\bm{x})) =  g^TMg
    \]
    a diagonal matrix. Thus $gv$ has no more mixed terms, confirming point (i) of the definition of a slice.
    
    A matrix $g\in \OO_n$ lies in the stabilizer $\BB$ if $g^T\text{diag}(\lambda_1,\dots,\lambda_n) g$ is again a diagonal matrix for any diagonal matrix $\text{diag}(\lambda_1,\dots,\lambda_n)$. This is the case if and only if $g$ is a matrix of orthonormal eigenvectors for $\text{diag}(\lambda_1,\dots,\lambda_n)$. If the eigenvalues $\lambda_i$ are distinct, then the only norm $1$ eigenvectors are $\pm e_i$ for $e_i$ the standard basis. These are exactly the signed permutation matrices and they clearly leave any diagonal matrix diagonal. This shows that the stabilizer of this slice $\BB = \BB_n$ is indeed given by the signed permutation matrices.
    
    We need to confirm condition (ii) of the definition of a slice. Note that the same reasoning as above applies to the stabilizer $\BB_n(\CC)$ of the subspace $\Lambda_2^n(\CC) \subset V_2^n(\CC) $ which shows $\BB_n = \BB_n(\CC)$. Thus if we take a point $s\in \Lambda_2^n(\CC)$ with distinct values for the $\lambda_1,\dots,\lambda_n$, then we know that $gs \in \Lambda_2^n(\CC)$ for some $g\in \OO_n$, implies that $g\in \BB_n$ already. The condition that the $\lambda_i$ be different is a general condition as it is the subset of $\Lambda_2^n(\CC)$ where the polynomial $\prod_{1\leq i<j\leq n}(\lambda_i-\lambda_j)$ does not vanish.

    Finally, for any positive degree $d$, we have the $\OO_n$-equivariant projection
    \[
    \pi: V_{2d}^n \rar V_2^n
    \]
    induced by the embedding $|\bm{x}|^{2d-2}V_2^n \subset V_{2d}^n$. Equivariance of the projection means that for any $g\in \OO_n$ and any $v\in V_{2d}^n$ we have $g\pi(v) = \pi(gv)$.
    Therefore the vector $v\in V_{2d}^n$ lies in the slice $\Lambda_{2d}^n$ if and only if $\pi(v)\in\Lambda_2^n$. This also holds for the complexification, thus the stabilizer of $\Lambda_{2d}^n$ is still the group of signed permutation matrices $\BB_n$.
    For the same reason, we see $\Lambda_{2d}^n$ satisfies conditions (i) and (ii) from Definition \ref{slice_def} and is indeed a slice. The isomorphism of rational invariants is now a direct consequence of Theorem \ref{slice_thm}.
\end{proof}

The isomorphism $\RR(V_{2d}^n)^{\OO_n} \cong \RR(\Lambda_{2d}^n)^{\BB_n}$ reduces our question to finding invariants for a finite group, a much simpler task.

The group of signed permutation matrices is a finite group with $2^n\cdot n!$ elements. It is also called the \emph{hyperoctahedral group}. We can write elements of $\BB_n$ uniquely as $\sigma\tau$, where $\tau \in \{\pm1\}^n$ and $\sigma \in S_n$ an element of the symmetric group. In matrix formulation $\tau = \text{diag}(\tau_1,\dots,\tau_n)$ with $\tau_i = \pm1$ and $\sigma$ is the unique $n\times n$ matrix with $\sigma e_i = e_{\sigma(i)}$ for all standard basis vectors $e_i$.

Thus, for a single $x_i$, the action of $\sigma\tau\in \BB_n$ is explicitly given by $\sigma\tau x_i = \tau_{i} x_{\sigma(i)}$. This extends multiplicatively to monomials in $V_{2d}^n$ and additively to all other elements. For a monomial $\bm{x}^\alpha = x_1^{\alpha_1} x_2^{\alpha_2} \dots x_n^{\alpha_n}$, the action is thus explicitly given by
\begin{equation}\label{explicit_action}
    \sigma\tau \bm{x}^\alpha = \tau^{\alpha} \bm{x}^{\sigma(\alpha)}
\end{equation}
where $\tau^\alpha = \tau_1^{\alpha_1}\tau_2^{\alpha_2}\dots \tau_n^{\alpha_n}$ and $\sigma(\alpha) = (\alpha_{\sigma^{-1}(1)},\alpha_{\sigma^{-1}(2)},\dots,\alpha_{\sigma^{-1}(n)})$ .

This gives in particular the action of $\BB_n$ on the slice $\Lambda_{2d}^n$. Note that $-\id\in \BB_n$ acts trivially  on $\Lambda_{2d}^n$; the slice is not a faithful $\BB_n$ representation.

\subsection{Symmetry adapted basis}
Our slice consists of two parts which behave differently and need different treatment. Let
\[
    W:= \{ v\in V_{2d}^n\mid \Delta^{d-1}(v) = 0\}.
\]
It is the subspace of all polynomials in $V_{2d}^n$ which are mapped to zero after applying the Laplacian $(d-1)$-times.
With Lemma \ref{harmonic_laplacian} it is immediate that
\[
W = |\bm{x}|^{2d-4}\HH_{4}^n \oplus \dots \oplus |\bm{x}|^2\HH_{2d-2}^n \oplus \HH_{2d}^n,
\]
and we have a decomposition into $\OO_n$ invariant subspaces $V_{2d}^n = |\bm{x}|^{2d-2}V_2^n \oplus W$. In particular, we have a decomposition of the slice into $\BB_n$ invariant subspaces
\[
\Lambda_{2d}^n = |\bm{x}|^{2d-2}\Lambda_2^n\oplus W.
\]

A \emph{symmetry adapted} or \emph{equivariant} basis for $|\bm{x}|^{2d-2}\Lambda_2^n $ is easily given by $m_i := |\bm{x}|^{2d-2}x_i^2$. This basis is permuted by $\BB_n$-action, ignoring the sign changes. We need to find a basis for $W$ with the related property that the basis elements are only permuted and negated by $\BB_n$-action.

Note that monomials form an equivariant basis for $V_{2d}^n$ under $\BB_n$-action. We can extend this idea to the subspace $W$ in the following way: The image of $\Delta^{d-1}$ is a degree two polynomial, thus if a monomial has more than two odd exponents, it already lies in the kernel. Otherwise, we can consistently pick another preimage of the image and subtract it to obtain an element in $W$. Since linear independence is clear by the uniqueness of the appearing monomials, a simple dimension counting argument shows that these elements form a basis.

For the rest of the article, $\alpha,\beta$ will be multi-indices in $\NN^n$ with entries summing to $2d$, except if explicitly stated otherwise. On such multi-indices, we have component-wise addition and, if the result is still in $\NN^n$, subtraction. Multiplication with a natural number is given by component-wise multiplication, and division by a natural number works analogously if all entries of the multi-index are divisible by the number.

\begin{definition}
    For a multi-index $\alpha$ with $\sum_{i=1}^n\alpha_i = m$, we define the multi-factorial as
    \[
        \alpha! := \prod_{i=1}^n \alpha_i!
    \]
    and the multinomial as
    \[
        \binom{m}{\alpha} := \frac{m!}{\alpha!}.
    \]
    We write $\epsilon_i\in\NN^n$ for the multi-index with
    \[
        (\epsilon_i)_j =
        \begin{cases}
        1,& \text{if } i = j\\
        0, & \text{if } i \neq j
        \end{cases}\ .
    \]
\end{definition}

\begin{definition}\label{def_eq_basis}
    As usual, we write $\bm{x}^\alpha$ for the monomial $\prod_{i=1}^n x_i^{\alpha_i}$. 
    \begin{enumerate}[label=(\roman*)]
        \item If $\alpha$ has only even entries, let
        \[
        m_\alpha := \binom{2d}{\alpha}\bm{x}^\alpha - \sum_{\alpha_i \neq 0} \binom{d-1}{\frac{\alpha}{2}-\epsilon_i}x_i^{2d}.
        \]
        \item If $\alpha$ has exactly two odd entries $\alpha_i$ and $\alpha_j$, let
        \[
        m_\alpha := \binom{2d}{\alpha}\bm{x}^\alpha - d\binom{d-1}{\frac{\alpha-\epsilon_i-\epsilon_j}{2}} \bigg( x_i^{2d-1}x_j + x_ix_j^{2d-1}\bigg)
        \]
        \item If $\alpha$ has more than two odd entries, let
        \[
        m_\alpha := \bm{x}^\alpha.
        \]
    \end{enumerate}
\end{definition}
Note that if $\alpha_i = 2d$ for any index, then $m_\alpha = 0$. If $\alpha_i = 2d-1$ and $\alpha_j = 1$, we write $m_\alpha =: m_{i,j} $. We have the relation $m_{i,j} = -m_{j,i}$. All other $m_\alpha$ are linearly independent which leads to the following.

\begin{proposition}\label{eq_basis}
    The polynomials $m_\alpha$ form an equivariant basis for $W$ for $\alpha\in\NN^{n}$ a multi-index summing to $2d$ with the properties
    \begin{enumerate}[label=(\roman*)]
        \item We have $\alpha_i\neq 2d$ for all $i=1,\dots,n$
        \item If there exists $i,j$ with $\alpha_i = 2d-1$ and $\alpha_j = 1$, then $i<j$.
    \end{enumerate}
\end{proposition}

\begin{proof}
    Explicitly computing the $d-1$-th power of the Laplacian gives
    \[
        \Delta^{d-1} = \sum_{\beta_1 + \dots + \beta_n = d-1} \binom{d-1}{\beta} \frac{\partial^{2d-2}}{\partial\bm{x}^{2\beta}}. 
    \]
    Applying each term of the sum to some monomial $\bm{x}^\alpha$ for some multi-index $\alpha$, we find
    \[
    \frac{\partial^{2d-2}}{\partial\bm{x}^{2\beta}}(\bm{x}^\alpha) =
    \begin{cases}
    \frac{\alpha!}{(\alpha-2\beta)!}\bm{x}^{\alpha-2\beta},& \text{if } \alpha_i \geq 2\beta_i \text{ for all $i = 1,\dots,n$}\\
    0,              & \text{otherwise}
    \end{cases}\ .
    \]
    
    Now, let $\alpha$ be a multi index summing to $2d$. If more than two $\alpha_i$ are odd, then the above shows that $\Delta^{d-1}(\bm{x^\alpha}) = 0$.

    If exactly $\alpha_i$ and $\alpha_j$ of $\alpha$ are odd, then
    \[
    \Delta^{d-1}(\bm{x}^\alpha) = \alpha! \binom{d-1}{\frac{\alpha-\epsilon_i-\epsilon_j}{2}} x_ix_j.
    \]
    Thus
    \[
    \Delta^{d-1}(m_\alpha) = (2d)!\binom{d-1}{\frac{\alpha-\epsilon_i-\epsilon_j}{2}} x_ix_j - 2d(2d-1)!\binom{d-1}{\frac{\alpha-\epsilon_i-\epsilon_j}{2}} x_ix_j = 0.
    \]
    If all entries of $\alpha$ are even, then
    \[
    \Delta^{d-1}(\bm{x}^\alpha) = \sum_{\alpha_i \neq 0} \frac{\alpha!}{2}\binom{d-1}{\frac{\alpha}{2}-\epsilon_i}x_i^2.
    \]
    Thus $\Delta^{d-1}(m_\alpha) = 0$, since
    \[
    \Delta^{d-1}(x_i^{2d}) = \frac{(2d)!}{2}x_i^2.
    \]
    We have shown $m_\alpha \in W$ for each $\alpha$ summing to $2d$. By our construction, for each $i = 1,\dots,n$, the exponents of $x_i$ appearing in each monomial of $m_\alpha$ have the same parity. Hence negating a single variable will either negate all terms of $m_\alpha$ or none. Permuting the variables by a $\sigma\in S_n$ gives $\sigma\cdot m_\alpha = m_{\sigma(\alpha)}$, where $\big(\sigma(\alpha)\big)_i = \alpha_{\sigma^{-1}(i)}$. So we have the action of $\BB_n$ explicitly given by $\sigma\tau (m_\alpha) = \tau^\alpha m_{\sigma(\alpha)}$ exactly as in \ref{explicit_action}. This shows that the $m_\alpha$ form an equivariant set under $\BB_n$ action.
    
    Under the above conditions on $\alpha$, the monomial $\bm{x}^\alpha$ appears exclusively in the polynomial $m_\alpha$ if $\alpha$ is not of the form $\alpha_i = 2d-1$ and $\alpha_j = 1$.
    Thus any linear dependence relation can only be among the $m_{i,j}$. But the condition $i<j$ guarantees that among the $m_{i,j}$ the monomial $x_i^{2d-1}x_j$ appears exclusively in $m_{i,j}$. 

    There are $\binom{n+2d-1}{n-1}$ multi-indices summing to $2d$ of which we have excluded $n+\binom{n}{2} = \binom{n+1}{2}$ by condition $(i)$ and $(ii)$.
    We know that $\dim W = \dim V_{2d}^n-\dim V_2^n = \binom{n+2d-1}{n-1} - \binom{n+1}{2}$. Thus the $m_\alpha$ form a maximal set of linearly independent elements in $W$ and are indeed a basis.
\end{proof}

\subsection{Orbit-type decomposition}
From the proof above, we see that the action of $\sigma\tau\in\BB_n$ on a basis element $m_\alpha$ maps it to $\pm m_{\sigma(\alpha)}$. 

\begin{definition}
    Let $\alpha,\beta \in \NN^n$. We write $\alpha \sim \beta$ if there exists a $\sigma \in S_n$ such that $\alpha = \sigma(\beta)$ and say that $\alpha$ and $\beta$, and respectively $m_\alpha$ and $m_\beta$ belong to the same orbit type.  
\end{definition}
This defines an equivalence relation on the set of multi-indices. A complete set of representatives is given by all multi-indices $\alpha$ with entries in descending order $\alpha_i \geq \alpha_j$ for $i\leq j$. We say that $\alpha$ is \emph{descending}.

\begin{definition}
    For a multi-index $\alpha\in \NN^n$ summing to $2d$, let
    \[
    W_\alpha := \langle m_\beta: \beta \sim \alpha\rangle_\RR
    \]
    be the subspace of $W$ spanned by all $m_\beta$ of the same orbit type as $m_\alpha$. 
\end{definition}

Each $W_\alpha$ is a $\BB_n$ subrepresentation of $W$ and we get a decomposition
\begin{equation}\label{orbit_type_decomp}
    W = \bigoplus W_\alpha
\end{equation}
where we sum over all descending $\alpha \in \NN^n$ summing to $2d$.

Fix a descending $\alpha$ that is not $(2d-1,1,0,\dots,0)$ or $(2d,0,\dots,0)$. Then the set of all $m_\beta$ of the same orbit type as $m_\alpha$ is already a basis for $W_\alpha$ by Proposition \ref{eq_basis}.
Fixing this equivariant basis allows us to describe its coordinate ring of $W_\alpha$ in terms of the corresponding coordinate functions. We write $c_\alpha$ for the coordinate function corresponding to $m_\alpha$ and $c_i$ for the coordinate function corresponding to $m_i$. Recall that these are defined by $c_\alpha(\sum_{\beta\sim\alpha} r_\beta m_\beta) = r_\alpha$ for a vector $\sum_{\beta\sim\alpha} r_\beta m_\beta\in W_\alpha$ with $r_\beta \in \RR$. This makes it clear that the coordinate functions also form an equivariant set under the action of $\BB_n$ and $\RR[{W_\alpha}] = \RR[c_\beta\mid \beta\sim\alpha]$, respectively $\RR[|\bm{x}^{2d-2}|\Lambda_2^n] = \RR[c_i \mid i = 1,\dots,n]$. In particular, $\RR(W_\alpha)^{\BB_n} = \RR(c_\beta\mid\beta\sim\alpha)^{\BB_n}$.

\section{The space of weighted graphs}\label{sec4}

In this section, we will deal with a generating set for the rational invariants of the subspace $W_\alpha \subset \Lambda_{2d}^n$ for $\alpha = (2d-1,1,0,\dots,0)$ for $n\geq 3$. This is not only an interesting question connected to graph theory in its own right, but it also allows us to neatly compute the remaining invariants in section \ref{sec5}.

\subsection{Connection to graph theory}
Let $\Gamma_n$ be the affine space of weighted graphs on $n$ vertices. We say that a weight of $0$ corresponds to an absent edge, and have $\Gamma_n \cong \bA^{\binom{n}{2}}$ as such a graph is determined by the weights of its $\binom{n}{2}$ edges $E_n = \big\{\{1,2\},\{1,3\},\dots,\{n-1,n\}\big\}$. We write
\[
\RR[\Gamma_n] = \RR[e_{i,j}: 1\leq i < j \leq n]
\]
for the coordinate ring of $\Gamma_n$. Here $e_{i,j}$ is the coordinate function which corresponds to the weight of the edge $\{i,j\}$ if $i< j$. Concretely, for some graph $a\in \Gamma_n$, $e_{i,j}(a)$ is the weight of edge $\{i,j\}$ of $a$. For $i < j$, we add new linearly dependent variables $e_{j,i} = -e_{i,j}$ for easier notation. We define the action of $\BB_n$ on edge weights by 
\[
\sigma\tau \cdot e_{i,j} = e_{i,j}\circ(\sigma\tau)^{-1} =  \tau_{i}\tau_{j} e_{\sigma(i),\sigma(j)}.
\]

This makes it clear that we have an isomorphism of $\BB_n$ representations, $W_\alpha \cong \Gamma_n$ for $\alpha = (2d-1,1,0,\dots,0)$ given by mapping $m_{i,j}$ for $i<j$ to the weight of the edge connecting $i$ and $j$. In particular, $e_{i,j}$ is the coordinate function corresponding to $m_{i,j}$ for both $i < j$ and $i> j$.

To use Lemma \ref{general-orbit-sep}, we will work with the complexification $\Gamma_n(\CC)$, which simply corresponds to the space of graphs with complex weights. We denote its ring of rational invariants as usual by $\CC(\Gamma_n)^{\BB_n}$.

Before we come to the main proposition of this section, we illuminate an interesting connection to the well-known graph isomorphism problem in computer science. This connection explains why we can only expect to separate general orbits with the generating set from Proposition \ref{weighted_graph_gen_set} and where the difficulties with separating all orbits come from. 

The graph isomorphism problem asks for an efficient algorithm that checks whether two given graphs lie in the same $S_n$ orbit acting on vertices.
This is a classical problem in complexity theory and relevant for the P versus NP problem. It is not an NP-hard problem but no polynomial time algorithm is known. This makes it a possible candidate for an intermediate complexity class. In a recent breakthrough \cite{Bab16} proved that the problem is solvable in quasi-polynomial time for any graph. In big-$O$ notation, this means there exists a constant $c$ such that $2^{O\big((\log n)^c\big)}$ computational steps are needed for a graph with $n$ vertices.

The graph isomorphism problem is concerned with classical, unweighted graphs, but it is not hard to see that it can be easily emulated if we restrict ourselves to the set of weighted graphs $H\subset \Gamma_n$ with edge weights either $0$ or $1$. With a corresponding $S_n$-action, a generating set $F_1,\dots,F_s$ for the invariant algebra $\CC[\Gamma_n]^{S_n} = \CC[F_1,\dots,F_s]$ also separates non-isomorphic graphs in $H$ as we have shown in Lemma \ref{orbit-sep}.

Polynomials can be evaluated efficiently, thus if we can prove the existence of a generating set of size $s = O(n^k)$ for some $k\in\NN$ with bounded degree, we would have shown that the graph isomorphism problem lies in the complexity class $P$. This indicates that this is a very hard problem.

Classical results like Noether's bound only give us an exponential bound. We can bound the degree of the generators by $|S_n| = n!$ which does not lead to any sensible bound on the number of generators $s$. 
For more about the connection of invariant theory and graph theory, see \cite[Chapter 5.5]{DK15}. The concrete problem of finding generating invariants $F_1,\dots,F_s$ for small $n$ is addressed in \cite{Thi00}, as well as a much better bound for $s$.

In contrast to these difficulties in finding a generating set for the ring of invariants, we only need to find a generating set for the field of rational invariants, which only separates general orbits. It turns out that the rational invariants we construct fail to separate orbits for any graph in $H$ by a wide margin. We will see this by constructing the open set on which orbit separation is successful in Proposition \ref{weighted_graph_gen_set}.

\subsection{Signed permutations on weighted graphs}
Let $\BB = \BB_n$ be the usual group of signed permutations of vertices and $\mathscr{B} = \BB(E_n)$ the group of signed permutations of edges.
As usual, we write elements of $\BB$ in two parts as $\sigma\tau$ with $\tau \in \{\pm1\}^n$ and $\sigma\in S_n$.
Correspondingly, we write elements of $\mathscr{B}$ as $\ul{\sigma\tau}$ with $\ul{\tau} = (\pm 1)_{1\leq i < j\leq n}$ and $\ul{\sigma}\in \mathrm{Aut}(E_n) \cong S_{\binom{n}{2}}$. 

The group $\BB$ acts on edge weights by 
\[
\sigma\tau \cdot e_{i,j} = e_{i,j}\circ(\sigma\tau)^{-1} =  \tau_{i}\tau_{j} e_{\sigma(i),\sigma(j)},
\]
with $e_{i,j} = -e_{j,i}$
and $\mathscr{B}$ acts on $\RR[\Gamma_n]$ by
\[
\ul{\sigma\tau} \cdot e_{i,j} := \ul{\tau}_{i,j} e_{\sigma(i,j)}.
\]

\begin{definition}
    We call a signed permutation $\ul{\sigma\tau} \in \mathscr B$ \emph{admissible} if there exists a $\sigma\tau \in \BB$ such that $\sigma\tau \cdot e_{i,j} = \ul{\sigma\tau} \cdot e_{i,j}$ for all edges $\{i,j\}\in E_n$.
\end{definition}

The permutation part of the action of $\sB$ ignores the orientation of edges, while the permutation part of the action of $\BB$ does not. We need to account for this when comparing the two actions. We call a set of three distinct vertices $\{i,j,k\}$ a triangle. Understanding the action of $\BB$ and $\sB$ on the corresponding triangles of edges $e_{i,j}$, $e_{i,k}$ and $e_{j,k}$ is what is required. Thus, for an ordered list of three distinct vertices $i,j,k$, we construct two polynomials

\begin{equation}\label{def_triangle_polys}
\begin{aligned}
    \theta_{i,j,k} &:= (e_{i,k}^2-e_{j,k}^2)(e_{i,j}^2-e_{j,k}^2)(e_{i,j}^2-e_{i,k}^2) \\
    \eta _{i,j,k} &:= e_{i,j}e_{i,k}e_{j,k}
\end{aligned}.
\end{equation}

\begin{lemma}\label{action_on_triangles}
    Let $\{i,j,k\}$ be any triangle of vertices. The following four properties hold:
    \begin{enumerate}[label=\arabic*.]
        \item For $\sigma$ a permutation of $\{i,j,k\}$, we have
        \begin{enumerate}[label*=\arabic*.]
            \item $\sigma\cdot \theta_{i,j,k} = \sgn(\sigma)\theta_{i,j,k}$
            \item $\sigma\cdot \eta_{i,j,k} = \sgn(\sigma)\eta_{i,j,k}$
        \end{enumerate}
        \item For $\sigma\in S_n$, we have
        \begin{enumerate}[label*=\arabic*.]
            \item $\sigma\cdot\theta_{i,j,k} = \theta_{\sigma(i),\sigma(j),\sigma(k)}$
            \item $\sigma\cdot\eta_{i,j,k} = \eta_{\sigma(i),\sigma(j),\sigma(k)}$
        \end{enumerate}
    \end{enumerate}
\end{lemma}

\begin{proof}
    For $1.1.$ we check
    \[
    \theta_{j,i,k} = (e_{j,k}^2-e_{i,k}^2)(e_{j,i}^2-e_{i,k}^2)(e_{j,i}^2-e_{j,k}^2) =-(e_{i,k}^2-e_{j,k}^2)(e_{i,j}^2-e_{j,k}^2)(e_{i,j}^2-e_{i,k}^2) =-\theta_{i,j,k}
    \]
    and by symmetry also $\theta_{i,k,j} = -\theta_{i,j,k}$. This suffices, since the two transpositions $(ij)$ and $(jk)$ generate all permutations of elements of $\{i,j,k\}$.
    Equivalently for $1.2.$, we check
    \[
    \eta_{j,i,k} = e_{j,i}e_{j,k}e_{i,k} = - e_{i,j}e_{i,k}e_{j,k} = -\eta_{i,j,k}
    \]
    and $\eta_{i,k,j} = - \eta_{i,j,k}$ which gives the proof for all permutations of $\{i,j,k\}$.

    Property $2.1.$ and $2.2.$ are clear from the definition in \ref{def_triangle_polys} and the action of $\sigma$ on $e_{i,j}$.
\end{proof}

\begin{lemma} \label{admissibility_lemma}
    A signed edge permutation $\ul{\sigma\tau} \in \mathscr B$ is admissible if and only if for each triangle $\{i,j,k\}$ there exists a triangle $\{i',j',k'\}$ such that
    \[
    \ul{\sigma\tau}\cdot\eta_{i,j,k}\theta_{i,j,k}= \eta_{i',j',k'}\theta_{i',j',k'}.
    \]
\end{lemma}

\begin{proof}
    The only if part is a direct consequence of Lemma \ref{action_on_triangles}, since any $\tau\in \{\pm1\}^n$ acts trivially on both $\theta_{i,j,k}$ and $\eta_{i,j,k}$. We prove the if part. 
    
    Suppose we have a $\ul{\sigma\tau}\in \sB$ such that for every triangle $\{i,j,k\}$ there exists a triangle $\{i',j',k'\}$ with
    \[
    \ul{\sigma\tau}\cdot\eta_{i,j,k}\theta_{i,j,k}= \eta_{i',j',k'}\theta_{i',j',k'}.
    \]
    The polynomial $\eta_{i,j,k}\theta_{i,j,k}$ only consists of the variables $e_{i,j}$, $e_{i,k}$ and $e_{j,k}$.  The sign part $\ul{\tau}$ of $\ul{\sigma\tau}$ only changes signs of variables. We can therefore conclude that $\ul{\sigma}$ maps the triangle of edges $\{i,j\},\{i,k\},\{j,k\}$ to a triangle of edges $\{i',j'\},\{i',k'\},\{j',k'\}$. This holds for every triangle on the graph, which lets us construct a corresponding permutation $\sigma$ on vertices by letting $\sigma(i)$ be the unique element of the intersection
    \[
    \sigma(i) \in \bigcap_{j\neq i} \ul{\sigma}(\{i,j\}).
    \]
    We check that $\sigma$ is well defined and indeed bijective. Suppose the intersection is empty. Then there exists $j,k$ such that $\ul{\sigma}(\{i,j\})\cap \ul{\sigma}(\{i,k\}) = \varnothing$. This is impossible since $\ul{\sigma}$ maps triangles of edges to triangles of edges.
    If the intersection contains more than one element, then there must exist $\{i,j\}$, $\{i,k\}$ such that $\ul\sigma(\{i,j\}) = \ul\sigma(\{i,k\})$. This is impossible, since $\ul\sigma$ is a bijection of the set of edges onto itself.
    Suppose the, now well-defined, map of vertices $\sigma$ is not a bijection. Then there exists $i,j$ such that $\sigma(i) = \sigma(j)$. Thus, for any $k \neq i,j$, $\ul\sigma(\{i,j\})\cap\ul\sigma(\{i,k\})\cap\ul\sigma(\{j,k\})$ is nonempty. This is impossible, as $\ul\sigma$ maps triangles of edges to triangles of edges.

    We have proved that there exists a $\sigma\in S_n$ with $\ul{\sigma\tau}\cdot e_{i,j} = \pm \sigma\cdot e_{i,j}$ for any edge $\{i,j\}$. To complete the proof of admissibility of $\ul{\sigma\tau}$, we show existence of a fitting $\tau\in\{\pm1\}^n$ such that $\ul{\tau}_{i,j}= \tau_i\tau_j$ by induction over $n$.  For the base case $n=3$, let $\tau_1 = 1$ and pick $\tau_2 = \ul{\tau}_{1,2}$ and $\tau_3 = \ul{\tau}_{1,3}$.
    From $\ul{\sigma\tau}\cdot\eta_{1,2,3}\theta_{1,2,3} = \eta_{1,2,3}\theta_{1,2,3}$, we conclude that $\ul{\tau}_{1,2}\ul{\tau}_{1,3}\ul{\tau}_{2,3} = 1$ and therefore
    \[
    \ul{\tau}_{2,3} = \ul{\tau}_{1,2}\ul{\tau}_{1,3} = \tau_1^2\tau_2\tau_3 = \tau_2\tau_3
    \]
    as required.

    Assume we have found $\tau_k$ for $k\leq n-1$ such that $\ul{\tau}_{i,j} = \tau_i \tau_j$ for $i,j\leq n-1$. We pick $\tau_n = \pm1$ such that $\ul{\tau}_{1,n} = \tau_1\tau_n$.

    With the same reasoning as in the base case, we find that for each $1< i < n $, $\ul{\tau}_{1,i}\ul{\tau}_{1,n}\ul{\tau}_{i,n} = 1$. Thus, using the induction assumption, we can check
    \[
    \ul{\tau}_{i,n} = \ul{\tau}_{1,i}\ul{\tau}_{1,n} = \tau_i^2\tau_i\tau_n = \tau_i\tau_n
    \]
    as required.
\end{proof}

We fix a special polynomial
\[
q := \sum_{1\leq i < j < k \leq n} \eta_{i,j,k}\theta_{i,j,k}
\]
with a particularly useful invariance condition.

\begin{proposition}\label{q_invariance_cond}
    An edge permutation $\ul{\sigma\tau}\in\mathscr{B}$ is admissible if and only if it leaves $q$ invariant.
\end{proposition}

\begin{proof}
    Let $\ul{\sigma\tau}$ be admissible. Then we find a $\sigma\tau \in \BB$ such that $\ul{\sigma\tau}\cdot q = \sigma\tau\cdot q = \sigma\cdot q$, since $\tau$ acts trivially on all $\theta_{i,j,k}$ and $\eta_{i,j,k}$. The element $\sigma$ induces a permutation of the terms $\eta_{i,j,k}\theta_{i,j,k}$ of $q$ by Lemma \ref{action_on_triangles} and therefore $q$ remains invariant under the action of $\ul{\sigma\tau}$.
    
    Suppose $\ul{\sigma\tau}\cdot q = q$. By construction \ref{def_triangle_polys}, and after expanding the terms $\eta_{i,j,k}\theta_{i,j,k}$ we see that each monomial of $q$ consists of exactly three variables $e_{i,j}$ with nonzero exponent and these three always form a triangle of edges. Thus $\ul{\sigma\tau}\cdot q = q$ implies that $\ul{\sigma}$ maps triangles of edges to triangles of edges. Let $\{i,j,k\}$ be any triangle and $\{i',j',k'\}$ be the triangle it is mapped to by $\ul\sigma$. Then $\ul{\sigma\tau}\cdot\eta_{i,j,k}\theta_{i,j,k} = \pm\eta_{i',j',k'}\theta_{i',j',k'}$, but again since $\ul{\sigma\tau}\cdot q = q$, it must be $\ul{\sigma\tau}\cdot\eta_{i,j,k}\theta_{i,j,k} = \eta_{i',j',k'}\theta_{i',j',k'}$. Thus $\ul{\sigma\tau}$ is admissible by Lemma \ref{admissibility_lemma}.
\end{proof}

\begin{proposition} \label{weighted_graph_gen_set}
    The field of complex rational invariants $\CC(\Gamma_n)^{\BB}$ is generated by the even-degree symmetric power sums
    \[
    p_\ell := e_{1,2}^{2\ell} + e_{1,3}^{2\ell} + \dots + e_{n-1,n}^{2\ell}
    \]
    for $\ell = 1,\dots, \binom{n}{2}$ and the polynomial
    \[
    q := \sum_{1\leq i < j < k \leq n} \eta_{i,j,k}\theta_{i,j,k}
    \]
    and this generating set separates all orbits outside of the Zariski-closed 
    \[
    X := \{a\in  \Gamma_n(\CC)\mid (\ul{\sigma\tau}\cdot q - q )(a) = 0,\ \ul{\sigma\tau}\in\mathscr{B}\text{ not admissible}\}.
    \]
    In particular, the field of real rational invariants $\RR(\Gamma_n)^\BB$ is generated by these polynomials.
\end{proposition}

\begin{proof}
    We begin by showing invariance of the generators. The even-degree symmetric power sums are invariant under $\mathscr B$ action, so they are in particular invariant under $\BB$ action. Proposition \ref{q_invariance_cond} shows that $q$ is $\BB$ invariant.

    For the second step, we show that $p_\ell$ and $q$ separate general orbits in $\Gamma_n(\CC)$. By Lemma \ref{general-orbit-sep} this completes the proof. The first $\binom{n}{2}$ even degree power sums generate the invariant algebra $\CC[\Gamma_n]^{\mathscr B}$, so by Lemma \ref{orbit-sep}
    they separate all $\mathscr{B}$ orbits. 

    Let $p\in \CC[\Gamma_n]^\sB$ be any $\sB$ invariant polynomial function written in variables $e_{i,j}$. Each exponent of $e_{i,j}$ appearing in $p$ can only be even, since otherwise $\ul{\tau}\in \sB$ with $\tau_{i,j} = -1$ and $\tau_{i',j'} = 1$ otherwise, does not leave $p$ invariant.
    Hence $\CC[\Gamma_n]^\sB = \CC[e_{i,j}^2\mid \{i,j\}\in E_n]^{\text{Aut}(E_n)}$ where the right side of the equation is simply the ring of symmetric polynomials in the variables $e_{i,j}^2$. It is a well known fact that this ring is generated by the first $\# E_n = \binom{n}{2}$ power sums, see e.g. Proposition 1.1.2 in \cite{Stu08}.

    We know that $\ul{\sigma\tau}\cdot q - q$ is not the zero polynomial for any non-admissible $\ul{\sigma\tau}\in \mathscr{B}$ by Lemma \ref{q_invariance_cond}. Thus $\Gamma_n(\CC)\setminus X$ is a nonempty Zariski-open set. Let $a,b\in \Gamma_n(\CC)\setminus X$ be any points lying in the same $\mathscr{B}$ orbit with $q(a) = q(b)$. Then there exists a $\ul{\sigma\tau}\in\mathscr{B}$ with $\ul{\sigma\tau}(b) = a$, hence
    \[(\ul{\sigma\tau}\cdot q - q)(a) = 0.\]
    But by the definition of $X$, if $a$ lies in the vanishing of the polynomial $\ul{\sigma\tau}\cdot q - q$ then $\ul{\sigma\tau}$ must be admissible. In other words, $a$ lies in the same $\BB$ orbit as $b$.
\end{proof}

\section{Computing the remaining invariants}\label{sec5}
There is an efficient way to find a generating set of rational invariants of $\RR(\Lambda_{2d}^n)^{\BB_n}$ using equivariance in $\Gamma_n$. In contrast to the previous section, we will not explicitly write the invariants, but rather describe the way they are obtained. This is more practical and will help to inform the evaluation algorithm in the final section, though our method is essentially a rather cumbersome closed-form formula.

\subsection{Equivariant polynomials}
Our goal is, in effect, to find a subrepresentation of $\BB_n$ in the coordinate ring $\RR[\Gamma_n]$ that is isomorphic to $W_\alpha$. Specifically, we construct polynomial functions $u_\alpha,d_\alpha \in\RR[\Gamma_n]$ with $\sigma\tau\cdot u_{\alpha} = u_{\sigma(\alpha)}$ and $\sigma\tau\cdot d_\alpha = \tau^{\alpha}d_{\sigma(\alpha)}$. Recall that we defined coordinate functions $c_\alpha$ such that $\RR[W_\alpha] = \RR[c_\beta\mid \beta\sim\alpha]$. This allows us to construct invariants of the form
\[
\inv(\alpha,k) = u_{\alpha^1}^kd_{\alpha^1}c_{\alpha^1} + \dots + u_{\alpha^N}^kd_{\alpha^N}c_{\alpha^N}.
\]

We implicitly order all multi-indices in the same equivalence class as $\alpha^1,\dots,\alpha^N$. This is done purely for ease of notation and the concrete order is irrelevant. Note that $N = N(\alpha)$ depends on the orbit-type of the partition.

\begin{definition}
    A set of polynomials $u_\alpha\in \RR[\Gamma_n]$ defined for each multi-index $\alpha$ of a certain orbit-type satisfies condition $(*)$ if both
    \[
    \sigma\tau \cdot u_\alpha = u_{\sigma(\alpha)} \quad\text{and}\quad u_\alpha \neq u_{\beta} \text{ for } \alpha \neq \beta.
    \]
    A set of polynomials $d_\alpha\in \RR[\Gamma_n]$ defined for each multi-index $\alpha$ of a certain orbit-type satisfies condition $(**)$ if
    \[
    \sigma\tau \cdot d_\alpha = \tau^{\alpha} d_{\sigma(\alpha)} \quad\text{and}\quad d_\alpha \neq 0.
    \]
\end{definition}

A system of equivariant polynomials fulfilling $(*)$ or $(**)$ for a given orbit type is not unique. Our construction was chosen for simplicity.

\begin{definition}
    For a subset $I\subset \{1,\dots,n\}$ of more than one element, define
    \[
    u_I := \sum_{\substack{ i\leq j \\ i,j\in I}} e_{i,j}^2.
    \]
    If $I = \{i\}$, then we let
    \[
    u_I = u_i = \sum_{j\neq i} e_{i,j}^2.
    \]
    Let $\alpha$ be a decreasing multi-index and let $I_1,\dots,I_k$ be a partition of $\{1,\dots,n\}$ induced by equality of entries of $\alpha$. Then we define $u_\alpha$ by
    \[
    u_\alpha := \prod_{\ell=1}^k u_{I_\ell}^{\ell-1}.
    \]
    For any partition $\beta$ of the same orbit-type, there exists a permutation $\sigma$ such that $\sigma(\alpha) = \beta$. This lets us define $u_\beta = \sigma \cdot u_\alpha$.
\end{definition}
Note that $u_\beta$ is independent of the choice of permutation $\sigma$ with $\sigma(\alpha) = \beta$ and therefore well-defined. We show this in Lemma \ref{const_sat_cond}.

\begin{definition}
    Let $\alpha$ be any multi-index summing to $2d$ and let $J \subset \{1,\dots,n\}$ be the set of $i$ for which $\alpha_i$ is odd.
    We define
    \[
    d_\alpha := \prod_{\substack{i,j\in J \\ i < j}} e_{i,j}(u_i-u_j).
    \]
\end{definition}

\begin{example}
    Let $\alpha = (2,2,1,1,0)$. The induced partition is $I_1 = \{1,2\}$, $I_2 = \{3,4\}$, $I_3 = \{5\}$. Thus we have
    \[
    u_\alpha = e_{3,4}^2(e_{1,5}^2+e_{2,5}^2 + e_{3,5}^2 + e_{4,5}^2)^2.
    \]
    Since $\alpha_3$ and $\alpha_4$ are odd, we have $J = \{3,4\}$ and
    \[
    d_\alpha = e_{3,4}(e_{1,3}^2 + e_{2,3}^2 + e_{3,5}^2 - e_{1,4}^2 - e_{2,4}^2 - e_{4,5}^2 ).
    \]

    Let $\alpha = (5,3,3,3,0,0)$. The induced partition is $I_1 = \{1\}$, $I_2 = \{2,3,4\}$, $I_3 = \{5,6\}$. Thus
    \[
    u_\alpha = (e_{2,3}^2 + e_{2,4}^2 + e_{3,4}^2)e_{5,6}^4.
    \]
    Entries $\alpha_1,\alpha_2,\alpha_3$ and $\alpha_4$ are odd, thus we have
    \begin{gather*}
        d_\alpha = e_{1,2}e_{1,3}e_{1,4}e_{2,3}e_{2,4}e_{3,4}
        (e_{1,3}^2 + e_{1,4}^2 + e_{1,5}^2 + e_{1,6}^2 - e_{2,3}^2 - e_{2,4}^2 - e_{2,5}^2 - e_{2,6}^2)\\
        (e_{1,2}^2 + e_{1,4}^2 + e_{1,5}^2 + e_{1,6}^2 - e_{2,3}^2 - e_{3,4}^2 - e_{3,5}^2 - e_{3,6}^2)
        (e_{1,2}^2 + e_{1,3}^2 + e_{1,5}^2 + e_{1,6}^2 - e_{2,4}^2 - e_{3,4}^2 - e_{4,5}^2 - e_{4,6}^2)\\
        (e_{1,2}^2 + e_{2,4}^2 + e_{2,5}^2 + e_{2,6}^2 - e_{1,3}^2 - e_{3,4}^2 - e_{3,5}^2 - e_{3,6}^2)
        (e_{1,2}^2 + e_{2,3}^2 + e_{2,5}^2 + e_{2,6}^2 - e_{1,4}^2 - e_{3,4}^2 - e_{4,5}^2 - e_{4,6}^2)\\
        (e_{1,3}^2 + e_{2,3}^2 + e_{3,5}^2 + e_{3,6}^2 - e_{1,4}^2 - e_{2,4}^2 - e_{4,5}^2 - e_{4,6}^2).
    \end{gather*}
\end{example}

\begin{lemma}\label{const_sat_cond}
    The elements $u_{\alpha}$ are well-defined and satisfy $(*)$, and the $d_\alpha$ satisfy $(**)$.
\end{lemma}
\begin{proof}
    We begin with property $(*)$. It is clear that $\tau\cdot u_\alpha = u_\alpha$ and by construction $\sigma\cdot u_\alpha = u_{\sigma(\alpha)}$ for any $\sigma\tau \in \BB$. We only need to show that $u_\beta$ is well-defined for a non descending multi-index $\beta$ and $u_\alpha \neq u_\beta$ for $\alpha \neq \beta$.

    Let $\alpha$ be descending and suppose we have two permutations $\sigma, \sigma'$. 
    By the factorization $u_\alpha = u_{I_2}^{}u_{I_3}^2\dots u_{I_k}^{k-1}$ and $\sigma \cdot u_\alpha =  u_{\sigma(I_2)}^{}u_{\sigma(I_3)}^2\dots u_{\sigma(I_k)}^{k-1}$, we see that $\sigma \cdot u_\alpha = \sigma'\cdot u_\alpha$ if and only if $\sigma(I_\ell) = \sigma'(I_\ell)$ for all $\ell$. This is the case if and only if $\sigma(\alpha) = \sigma'(\alpha)$ and therefore confirms that all $u_\beta$ are well defined and $u_\beta \neq u_{\beta'}$ for $\beta \neq \beta'$.
    
    Take any $\alpha$ and $\sigma\tau \in \BB_n$ and let $J$ be the index set of odd entries in $\alpha$. Clearly $\tau^\alpha = \prod_{i\in J} \tau_i$, since only odd powers of $\pm1$ need to be considered, so we compute
    \begin{align*}
        \sigma\tau\cdot d_\alpha &=
        \prod_{\substack{i,j\in J \\ i \leq j}} \tau_i\tau_j e_{\sigma(i),\sigma(j)} (u_{\sigma(i)} - u_{\sigma(j)}) \\
        &=  \prod_{i\in J}\tau_i^{|J|-1} \prod_{\substack{i,j\in \sigma(J) \\ i \leq j}} e_{i,j} (u_{i} - u_{j}) \\
        &= \tau^\alpha d_{\sigma(\alpha)}.
    \end{align*}
    We used that $\sigma(J)$ are the odd indices of $\sigma(\alpha)$, $|J|$ is always even since $\sum_i \alpha_i$ is even, and $e_{i,j}(u_i - u_j) = e_{j,i}(u_j - u_i)$.
\end{proof}

\begin{definition} \label{inv_def}
    For a descending $\alpha$ we list all multi-indices of its orbit-type, $\alpha^1,\alpha^2,\dots,\alpha^N$. This lets us define the corresponding invariants
    \[
    \inv(\alpha,k) := u_{\alpha^1}^kd_{\alpha^1}c_{\alpha^1} + \dots + u_{\alpha^N}^kd_{\alpha^N}c_{\alpha^N}
    \]
    where $k$ ranges from $0$ to $N-1$.

    For a complete set we also define the invariants 
    \[
    \inv(0,k) = u_1^kc_1 + \dots +u_n^kc_n
    \]
    where $k$ ranges from $0$ to $n-1$.
\end{definition}

\begin{lemma}\label{proof_of_invariance}
    The elements $\inv(0,k)$ and $\inv(\alpha,k)$ for a descending $\alpha$ are invariant under $\BB_n$ action. 
\end{lemma}
\begin{proof}
    Recall that the action of $\sigma\tau\in \BB_n$ on $c_i$ is given by $\sigma\tau\cdot c_i = c_{\sigma(i)}$. Thus
    \[
    \sigma\tau\cdot\inv(0,k) = \sigma\tau \cdot(u_1^kc_1 + \dots + u_n^kc_n) = u_{\sigma(1)}^kc_{\sigma(1)} + \dots + u_{\sigma(n)}^kc_{\sigma(n)} = \inv(0,k).
    \]

    Fix a descending $\alpha$ and let $\alpha^1,\dots,\alpha^N$ be a list of all multi-indices of its orbit type. Then  $\sigma$ induces some permutation on the $u_{\alpha^1},\dots,u_{\alpha^N}$. Since $\tau\cdot c_{\alpha^i} = \tau^{\alpha^i}c_{\alpha^i}$ we have
    \begin{align*}
        \sigma\tau\cdot \inv(\alpha,k) &=
        \sigma\tau \cdot(u_{\alpha^1}^kd_{\alpha^1}c_{\alpha^1} + \dots + u_{\alpha^N}^kd_{\alpha^N}c_{\alpha^N}) \\
        &= \tau^{2\alpha^1}u_{\sigma(\alpha^1)}^k d_{\sigma(\alpha^1)} c_{\sigma(\alpha^1)} + \dots + \tau^{2\alpha^N} u_{\sigma(\alpha^N)}^k d_{\sigma(\alpha^N)} c_{\sigma(\alpha^N)} \\
        &= \inv(\alpha,k),
    \end{align*}
    since $\tau^{2\beta} = 1$ for any $\beta$ and $\sigma$ simply permutes the terms $u_{\alpha^i}^k d_{\alpha^i} c_{\alpha^i}$.
\end{proof}

\begin{theorem} \label{main_thm}
    A generating set for the field of rational invariants $\RR(\Lambda_{2d}^n)^{\BB_n}$ is given by a generating set for the field of rational invariants $\RR(\Gamma_n)^{\BB_n}$ and the invariants $\inv(\alpha, k)$ and $\inv(0,k)$ defined above.
\end{theorem}
\begin{proof}
     By Lemma \ref{proof_of_invariance}, the given set indeed consists of invariants. It suffices to show that every $f\in \RR[\Lambda_{2d}^n]^{\BB_n}$ can be written as a rational function in the given invariants. This is due to Lemma \ref{fininte-group-frac-field}, which shows that the rational invariants $\RR(\Lambda_{2d}^n)^{\BB_n}$ are the fraction field of the ring of invariants $\RR[\Lambda_{2d}^n]^{\BB_n}$ as $\BB_n$ is a finite group.

    Fix a descending $\alpha$. We define the Vandermonde matrices
    \[
    M(\alpha) :=
    \begin{pmatrix}
    1 & u_{\alpha^1} & \cdots & u_{\alpha^1}^{N-1}\\
    \vdots & \vdots & \ddots & \vdots\\
    1 & u_{\alpha^N} & \cdots & u_{\alpha^N}^{N-1}
    \end{pmatrix}
     \qquad \text{and}\qquad
     M(0) :=
    \begin{pmatrix}
    1 & u_{1} & \cdots & u_{1}^{n-1}\\
    \vdots & \vdots & \ddots & \vdots\\
    1 & u_{n} & \cdots & u_{n}^{n-1}
    \end{pmatrix}
    \]
    as well as the diagonal matrices $D(\alpha) = \text{diag}(d_{\alpha^1},\dots,d_{\alpha^N})$.
    Let
    \[
    \bm{c}(\alpha) := \begin{pmatrix} c_{\alpha^1} \\ \vdots \\ c_{\alpha^N} \end{pmatrix}
    \quad\text{and}\quad
    \bm{c}(0) := \begin{pmatrix} c_{1} \\ \vdots \\ c_{n} \end{pmatrix},
    \]
    then
    \[
    M(\alpha)^T D(\alpha) \bm{c}(\alpha) = \begin{pmatrix} \inv(\alpha,0) \\ \vdots \\ \inv(\alpha,N-1) \end{pmatrix}
    \quad\text{and}\quad
    M(0)^T\bm{c}(0) = \begin{pmatrix} \inv(0,0) \\ \vdots \\ \inv(0,n-1) \end{pmatrix}.
    \]

    The matrices $M(\alpha)^T, M(0)^T$ and $D(\alpha)$ are all invertible over the field $\RR(\Gamma_n)$. We have $0 \neq \prod d_{\alpha^i}\in \RR(\Gamma_n)$ and $\det M(\alpha) = \prod_{i<j}(u_{\alpha^i} - u_{\alpha^j}) \in \RR(\Gamma_n)$ which is nonzero by condition $(*)$. The same argument works for $M(0)^T$.

    We can therefore write each $c_i$ and $c_{\alpha}$ as a $\RR(\Gamma_n)$-linear combination of the invariants $\inv(0,k)$ and $\inv(\alpha,k)$ given by the entries of
    \[
    D(\alpha)^{-1}({M(\alpha)^T})^{-1} \begin{pmatrix} \inv(\alpha,0) \\ \vdots \\ \inv(\alpha,N-1)\end{pmatrix} =
    \begin{pmatrix} c_{\alpha^1} \\ \vdots \\ c_{\alpha^N} \end{pmatrix}.
    \]
    This clearly holds for every descending $\alpha$. Therefore every $f\in R[\Lambda_{2d}^n]^{\BB_n}$ can be rewritten as a polynomial in the invariants $\inv(\alpha,k)$ and $\inv(0,k)$ with coefficients in the field $\RR(\Gamma_n)$. Since the polynomial $f$ is invariant and the $\inv(\alpha,k)$ and $\inv(0,k)$ are algebraically independent and invariant, we find that the coefficients of this polynomial must be invariant as well. Hence they can be written as a rational function in the generating set of $\RR(\Gamma_n)^{\BB_n}$. This gives a complete description of $f$ only in terms of the chosen invariants.
\end{proof}

Using Theorem \ref{main_thm}, Theorem \ref{weighted_graph_gen_set}, and the slice method from Theorem \ref{slice_thm}, we have the following result.

\begin{corollary}\label{main_corollary}
    A generating set $\cJ_{2d}^n \subset \RR(\Lambda_{2d}^n)^{\BB_n}$ is given by the invariants
    \[
    q, p_1, \dots,p_{\binom{n}{2}}, \inv(0,0),\dots,\inv(0,n-1),
    \]
    and $\inv(\alpha,0),\dots,\inv(\alpha,N-1)$ for every descending $\alpha\in \NN^n$ summing to $2d$. 

    A generating set $\cI_{2d}^n \subset \RR(V_{2d}^n)^{\OO_n}$ is given by its restriction, $f\in \cI_{2d}^n$ if and only if $f\rest{\Lambda_{2d}^n} \in \cJ_{2d}^n$.
\end{corollary}

\section{Algorithmic evaluation} \label{sec6}
In Sections \ref{sec4} and \ref{sec5} we have explicitly constructed a generating set $\mathcal{J}_{2d}^n\subset \RR(\Lambda_{2d}^n)^{\BB_n}$. This generating set even consists of polynomial functions. Listing the invariants as $f_1,\dots,f_s \in \mathcal{J}_{2d}^n$, we identify a generating set of rational functions $\{\tilde f_1, \dots, \tilde f_s \} = \mathcal{I}_{2d}^n \subset \RR(V_{2d}^n)^{\OO_n}$ using the slice method \ref{slice_thm}. Concretely the rational functions are given by
\[
 \tilde f_i(v) := f_i(gv) \text{ for any } g \in \OO_n \text{ such that } gv \in \Lambda_{2d}^n,   
\]
for any $v\in V_{2d}^n$ where $\tilde f_i$ is defined.
Algorithm \ref{eval_alg} shows how to evaluate these.

\begin{algorithm}[h]
    \caption{Evaluation Algorithm}
    \label{eval_alg}
    \begin{flushleft}
        \textbf{Input:} A polynomial $v\in V_{2d}^n$. \\
        \textbf{Output:} Evaluation of rational invariants of $V_{2d}^n$ on $v$ if defined.  
    \end{flushleft}
    \begin{algorithmic}[1]
        \State Compute the decomposition $v = |\bm{x}|^{2d-2}v' + w$ where $v' \in V_{2}^n$ and $\Delta^{d-1}(w) = 0$.
        \State Find $g\in \OO_n$ s.t. $gv' = \lambda_1x_1^2 + \dots + \lambda_nx_n^2$.
        \State If $\lambda_i = \lambda_j$ for $i \neq j$, output ``undefined at $v$", stop.
        \State Compute $r := gv \in \Lambda_{2d}^n$.
        \State Rewrite $r$ in equivariant basis $m_\alpha$, $m_i$ as in \ref{def_eq_basis}.
        \State Output $p^\ell(r)$ for $\ell = 1,\dots,\binom{n}{2}$, $q(r)$, $\inv(0,k)(r)$ for $k=0\dots,n-1$ and for each descending $\alpha$: $\inv(\alpha,k)(r)$ for $k= 0,\dots, N(\alpha)-1$.
    \end{algorithmic}
\end{algorithm}

The evaluation is defined whenever this algorithm gives an output, but it does not always separate orbits. For this, an additional check whether $r \rest{\Gamma_n}\in X$ is required for the Zariski-closed $X$ from Theorem \ref{weighted_graph_gen_set}.

Even though some invariant $\tilde f_i$ might not be defined at $v \in V_{2d}^n$, clearly there still exists a $g\in \OO_n$ with $gv \in \Lambda_{2d}^n$ and an evaluation of $f_i(gv)$ is possible. However, it is not well defined, as it is dependent on the choice of $g$. Correctness of this algorithm follows from Theorem \ref{main_thm} and the description of the isomorphism  $\RR(\Lambda_{2d}^n)^{\BB_n} \cong \RR(V_{2d}^n)^{\OO_n}$ from Theorem \ref{slice_thm}. 

We comment on the step by step implementation. Step $1$ is simply an application of the harmonic decomposition. Recall that $W := |\bm{x}|^{2d-4}\HH_{4}^n\oplus \dots \oplus \HH_{2d}^n \subset V_{2d}^n$ is exactly the space of polynomials $v\in V_{2d}^n$ with $\Delta^{d-1}(v) = 0$. We have the decomposition
\[
V_{2d}^n = |\bm{x}|^{2d-2} V_2^n \oplus W,
\]
and therefore unique $v'\in V_2^n$, $w\in W$ such that $v = |\bm{x}|^{2d-2}v' + w$ and $\Delta^{d-1}(w) = 0$. Computing $v'$ amounts to solving the linear system of equations induced by
\[
\Delta^{d-1}(|\bm{x}|^{2d-2}v') = \Delta^{d-1}(v).
\]

Step 2 is the most critical one. We move our polynomial onto the slice. Let $A\in \RR^{n\times n}$ be a symmetric $n\times n$ matrix corresponding to $v'$ by $\bm{x}^T A \bm{x} = v'$. Then the required $g$ is the orthogonal matrix such that $g^T A g$ is a diagonal matrix. Finding $g$ is the well-known eigendecomposition problem for a symmetric matrix. Numerical methods for this problem are well-known, see e.g. \cite[Chapter 8]{GVL13} .
It is also theoretically possible to find exact symbolic solutions by working in number fields, though such computations quickly become computationally intractable.

In Step 3, we simply check whether the rational invariants are defined at $v$ using the eigenvalues computed in the previous step. If so, we move the whole polynomial onto the slice by the action of $g$ in Step 4. The product $r = gv$ is easy to compute, since the action of $\OO_n$ on $V_{2d}^n$ is given by the action on variables. Rewriting it in our equivariant basis is a simple basis change operation. The basis change matrix from a monomial basis into $m_\alpha$, $m_i$ is given by the construction in Definition \ref{def_eq_basis}. Finally, Step 6 is a simple evaluation of multivariate polynomials at $r$. These are explicitly given in Theorem \ref{weighted_graph_gen_set} and Definition \ref{inv_def}.

Note that the value $N(\alpha)$ in Step 6 is the number of multi-indices of the same orbit type as $\alpha$. It is easy to compute. For a fixed $\alpha$, let $h_k$ be the number of times $k$ appears as an entry of $\alpha$ for $k = 0,\dots,2d$. Then $N(\alpha) = \frac{n!}{h_0!h_1!\dots h_{2d}!}$.

\section{An odd number of variables}\label{sec7}
For the final section, we explore some desirable properties of the rational invariants $\RR(V_{2d}^n)^{\OO_n}$ in the case of odd $n\geq 3$. In this restricted case, there is a trivial equivalence of the rational invariants of $\OO_n$ and $\SO_n$ and we can answer the question whether an algebraically independent generating set exists.

The rational invariants $\RR(V^n_{2d})^{\OO_n}$ can be generated by algebraically independent elements over $\RR$ if the field extension $\RR(V_{2d}^n)^{\OO_n}/\RR$ is \emph{purely transcendental}, which we also call \emph{pure}. A field extension $K/k$ is called \emph{stably pure}, if there exists a pure extension $\mathcal{K}/K$ such that $\mathcal{K}/k$ is also pure.

It is known, that all almost free, finite dimensional representations of $\OO_n(\RR)$ have stably pure rational invariant fields over $\RR$, see \cite[Proposition 4.10]{CTS07}. In our specific case, we can use the representation on the slice to abstractly show the stronger property of being generated by algebraically independent elements in Theorem \ref{odd_purety}.

\begin{proposition}\label{spec_orth_case}
    For odd $n\geq 3$, the field of rational invariants for $\OO_n$ and $\SO_n$ on $V_{2d}^n$ are equal,
    \[
    \RR(V_{2d}^n)^{\OO_n} = \RR(V_{2d}^n)^{\SO_n}.
    \]
\end{proposition}
\begin{proof}
    For $n\geq 3$, the center of $\OO_n$ is the order $2$ subgroup $Z(\OO_n) = \{\pm\id\}$. If $n$ is odd, then $\det(-\id) = -1$, thus we have a split exact sequence
    \[
    1 \rar \{\pm \id\} \rar \OO_n \rar \SO_n \rar 1
    \]
    and $\OO_n \cong \SO_n\times \{\pm \id\}$. Since $V_{2d}^n$ is the space of even degree polynomials $\{\pm\id\}$ induces the trivial action on $V_{2d}^n$. Thus any rational function $f\in \RR(V_{2d}^n)$ is $\OO_n$-invariant if it is $\SO_n$-invariant. 
\end{proof}

\begin{theorem}[Speiser's Lemma]\label{speisers_lemma}
    Let $k$ be any field with $\text{char k} = 0$. Let $G$ be a finite group and $V$, $W$ two finite dimensional, linear representations of $G$ over $k$. Suppose that $V$ is a faithful representation. Then $k(V\oplus W)^G/k
    (V)^G$ is a purely transcendental extension. 
\end{theorem}

This result is given by the proof of \cite[Theorem 3.3]{CTS07}. Note that in the reference, there is a generic assumption that $k$ is an algebraically closed field. However, this assumption is not used in the proof.
It can be seen as a version of the no-name-lemma \cite[Corollary 3.9]{CTS07} for finite groups. For the following corollary, see also \cite[Corollary 3.10]{CTS07}.

\begin{corollary}\label{purety_condition}
    Let $G$ be a finite group. Let $V_1,V_2$ and $W$ be three finite dimensional linear $G$ representations over $\RR$. Assume that $V_1$ and $W$ are faithful representations, $\RR(W)^G/\RR$ is pure and $\dim V_2 \geq \dim W$. Then $\RR(V_1\oplus V_2)^G/\RR$ is a pure extension.
\end{corollary}
\begin{proof}
    Denote $d_1 = \dim W$ and $d_2 = \dim V_2$. We use Theorem \ref{speisers_lemma} repeatedly. Since $W$ is faithful, $\RR(V_1\oplus W)^G/\RR(W)^G$ is purely transcendental. By our assumption, $\RR(W)^G/\RR$ is pure, so $\RR(V_1\oplus W)^G/\RR$ is pure as well. Since $V_1$ is faithful, $\RR(V_1\oplus W)^G/\RR(V_1)^G$ and $\RR(V_1\oplus V_2)^G/R(V_1)^G$ are pure. In other words $\RR(V_1\oplus W)^G \cong \RR(V_1)^G(y_1,\dots,y_{d_1})$ and $\RR(V_1\oplus V_2)^G \cong \RR(V_1)^G(y_1,\dots,y_{d_2})$ for formal variables $y_i$ with trivial $G$-action. But $d_1\leq d_2$ and so $\RR(V_1\oplus V_2)^G \cong \RR(V_1\oplus W)^G(y_{1},\dots,y_{d_2-d_1})$ is a pure extension of $\RR$.
\end{proof}

\begin{theorem}\label{odd_purety}
    The field extension $\RR(V_{2d}^n)^{\OO_n}/\RR$ is pure for odd $n\geq 3$.
\end{theorem}
\begin{proof}
    Let $n\geq3$ be odd. We can immediately reduce to the action of $\BB_n$ on the slice $\Lambda_{2d}^n$, since by Theorem \ref{slice_thm}, $\RR(V_{2d}^n)^{\OO_n} \cong \RR(\Lambda_{2d}^n)^{\BB_n}$.
    We need to deal with a special case before solving the general case.
    
    Suppose $d=1$. Then the action of a signed permutation matrix $\BB_n$ on the slice
    \[
    \Lambda_{2}^n = \langle x_i^2\mid i=1,\dots,n\rangle_\RR
    \]
    is just the action of the underlying permutation matrix without the signs. Thus the field of rational invariants for the $\BB_n$ representation $\Lambda_2^n$ is isomorphic to the field of rational invariants for the natural representation of $S_n$. This is well known to be a pure extension, generated for example by the elementary symmetric polynomials.  
    
    For the general case, let $n\geq 3, d\geq 2$. The action of $\BB_n$ on $\Lambda_{2d}^n$ is not faithful, since $-\id \in \BB_n$ but the action of the quotient $\BB_n/\{\pm\id\}$ is.
    
    Since $n$ is odd, $\det (-\id) = -1$, and we find that the quotient $\BB_n/\{\pm\id\} =: \DD_n$ is the group of signed permutation matrices with positive determinant. It is a Coxeter group, which gives us the following approach:
    
    Let $W_n := \RR^n$ with fixed standard basis. We realize $\DD_n \subset \GL(W_n)$ as the group of signed permutation matrices with positive determinant. This makes $W_n$ a faithful representation and $\DD_n$ is generated by its reflections in $W_n$. Thus $\RR[W_n]^{\DD_n}$ is a polynomial ring by the Chevalley-Shephard-Todd theorem. In particular, its fraction field $\RR(W_n)^{\DD_n}$ is pure of transcendence degree $n$ over $\RR$.
    
    We show that the $\DD_n$ subrepresentation $\Gamma_n \subset \Lambda_{2d}^n$ is faithful and has codimension at least $\dim W_n$. Then Corollary \ref{purety_condition} completes the proof.

    Take any $\sigma\tau\in \DD_n \subset \BB_n$ and suppose $\sigma\tau$ acts trivially on $\Gamma_n$. Then $\sigma\tau\cdot e_{i,j} = e_{i,j}$ for all $1\leq i<j\leq n$. Thus $\sigma$ either fixes or swaps $i$ and $j$. Since $\sigma$ fixes all the edges of any triangle $\{i,j,k\}$ it must fix $i$, $j$ and $k$. We therefore find $\sigma\tau\cdot e_{i,j} = \tau_i\tau_j e_{i,j}$, thus $\tau_i = \tau_j = \pm 1$. As this holds for all pairs $i,j$, we find either $\tau_i = 1$ for all $i = 1,\dots,n$ or $\tau_i = -1$ for all $i=1,\dots,n$. In other words $\sigma\tau$ is either $\id$ or $-\id$ in matrix notation. Since $\det(-\id) = -1$ we see that $\sigma\tau$ must be the identity matrix in $D_n$ and $W_n$ is indeed a faithful representation.

    The last part required for Corollary \ref{purety_condition} is a simple dimension counting argument. We know $\dim \Gamma_n = \binom{n}{2}$ and $\dim \Lambda_{2d}^n =\binom{n+2d-1}{2d}-\binom{n}{2}$. Clearly $\dim W_n = n$ and for $n\geq 3, d\geq 2$ we have
    \[
    \dim W_n = n\leq \dim \Lambda_{2d}^n - \dim \Gamma_n = \binom{n+2d-1}{2d}-2\binom{n}{2}.\qedhere
    \]
\end{proof}

The group $\BB_n/\{\pm\id\}$ is not a Coxeter group for even $n\geq 4$, so a simple rationality argument in the same way is not possible.

\printbibliography

\end{document}